\documentclass[a4paper,11pt, leqno]{article}
\usepackage[latin1]{inputenc}
\usepackage[hidelinks]{hyperref}
\usepackage{articlemacro}

\usepackage[style=alphabetic]{biblatex}
\bibliography{references}
\DeclareFieldFormat{postnote}{#1}
\DeclareFieldFormat{multipostnote}{#1}

\usepackage[total={17cm,24cm}]{geometry}

\begin{document}

\title{Perverse Sheaves on the Stack of \(G\)-Zips}
\author{Christopher Lang}

\hypersetup{
    pdftitle={Perverse Sheaves on the Stack of G-Zips}, % Title field in PDF metadata
    pdfauthor={Christopher Lang},       % Author field in PDF metadata
}

%==================================================================
\maketitle

\section{Introduction}
\addcontentsline{toc}{section}{Introduction}

If we have a Shimura variety \(\Scal\) of abelian type associated to a Shimura datum \((\GG,X)\), one can look at its reduction \(\Scal_0\) at a prime \(p\) of good reduction and get a smooth map 
\[\zeta: \Scal_0 \to G\Zip^\mu,\]
see \cite{shen2021stratificationsgoodreductionsshimura}. Here, \(G\) is the base change to \(\FF_q\) of an integral model of the adjoint of \(\GG\), and \(\mu\) is the reduction of a cocharacter from the Shimura datum \((\GG,X)\). The idea is that the stack of \(G\)-zips is easier to study than Shimura varieties. Let's explain why that is and let's give a rough definition of \(G\)-zips in the process.

If the Shimura variety is of PEL type, so that it is the moduli space of abelian varieties with additional structure, the map \(\zeta\) associates to an abelian variety its first de Rham cohomology, which naturally comes equipped with the Hodge filtration and the conjugate filtration, and we have Frobenius-linear isomorphisms between the graded pieces of the two filtrations. Such an object, namely a vector space with two filtrations and Frobenius-linear isomorphisms of the graded pieces, is called an F-zip. These are objects of linear algebra, and hence easier to understand than Shimura varieties. Now, a \(G\)-zip is just an F-zip that is equipped with a structure of the group \(G\). By that we mean an exact tensor functor from the category of representations of \(G\) to the category of F-zips. The connected components of the stack of \(G\)-zips are classified by conjugacy classes of cocharacters, and one such connected component is called \(G\Zip^\mu\), the \(G\)-zips of type \(\mu\).

The fibres of the map \(\zeta\) are the Ekedahl-Oort strata of the Shimura variety. Furthermore, the stack of \(G\)-zips of type \(\mu\) can be identified as a quotient stack \([E_\mu\backslash G]\), where \(E_\mu\) is a subgroup of \(G\times G\) acting on \(G\) by \((g_1,g_2).g=g_1gg_2\inv\). The underlying topological space of \(G\Zip^\mu\) can be explicitly computed using the Weyl group of \(G\), and the stabilizers of the quotient stack are semi-direct products of a finite group of Lie type with a connected unipotent group. 

The idea is now that if you are interested in some type of object over the Shimura variety, a good starting point can be to look at this type of object over the stack of \(G\)-zips and then pull it back to the Shimura variety. For example you can do this for the category of perverse sheaves: On the stack of \(G\)-zips of type \(\mu\), these are \(E_\mu\)-equivariant perverse sheaves on \(G\), and since the map \(\zeta\) is smooth, we obtain a shifted pullback functor to the category of perverse sheaves on the Shimura variety. 

\subsection*{Perverse Sheaves on \texorpdfstring{\(G\)}{G}-Zips}
We compute the simple perverse sheaves on the stack of \(G\)-zips of type \(\mu:\Gm \to G\) in several examples. It follows from the general theory that these can be computed using the following principle:
\begin{introtheorem}[Theorem \ref{thm_simp_perv}]
    Given a connected reductive group \(G\) over \(\FF_q\) and a cocharacter \(\mu:\Gm[\FFbar_q]\to G_{\FFbar_q}\), the simple perverse sheaves on the stack of \(G\)-zips of type \(\mu\) are given by pairs \((w,\theta)\), where \(w\) is an orbit of the \(E_\mu\) action on \(G\) and \(\theta\) is an irreducible \(\QQbar_l\)-representation of the group of connected components of the stabilizer of this group action at \(w\).
\end{introtheorem}
The set of orbits is \({}^IW\), where \(W\) is the Weyl group of \(G\) and \(I\) is the type of \(\mu\), and the group of connected components of the stabilizer of this group action at \(w\) is known to be 
\[\Pi_w\coloneqq \set{h\in H_w(\FFbar_p)}{h=\phi\big(\dot y \dot w h (\dot y \dot w)\inv\big)},\]
where \(H_w\) is some Levi-subgroup of \(G_{\FFbar_q}\) and \(y\in G\) depends only on \(I\). This can be understood as a group over \(\FFbar_p\) together with a twisted Frobenius action, which yields a descent datum to \(\FF_q\), and the group above are just the \(\FF_q\)-valued points thereof. Hence one needs to understand the irreducible \(\QQbar_l\)-representations of certain finite groups of Lie type, and these can in principle be understood with Deligne-Lusztig theory. 
We list some examples:
\begin{itemize}
    \item If one takes the constant cocharacter, there is only one orbit and the corresponding group is \(\Pi_w=G(\FF_q)\).
    \item For \(G=\GL_n\) and the cocharacter that corresponds to a Borel subgroup, the set of orbits is \(W=\S_n\), the group \(H_w\) is a torus and \(\Pi_w\) is a product of some unit groups of finite fields. In particular, these are abelian groups, so their irreducible representations can be written down explicitly.
    \item For the symplectic group \(\Sp_{2n}\) and the cocharacter that corresponds to the parabolic with two \(n\times n\)-blocks on the diagonal, the set of orbits is \(\{0,1\}^n\), and for the closed orbit \(w=\id\) we have \(\Pi_{\id} = \U_n(\FF_q)\), while for the open orbit we have \(\Pi_w = \GL_n(\FF_q)\). The other \(\Pi_w\) we computed for \(n=2\), where they are again abelian groups.
    \item If we take the general unitary group \(G=\GU_n\) with a cocharacter \(\mu\) of type \((1,n-1)\), then the set of orbits is totally ordered of size \(n\). For each orbit \(w\), the group \(\Pi_w\) is an extension of \(\GU_{n_w}(\FF_q)\) by some abelian group, where \(0\le n_w< n\) is an integer depending on \(w\).
\end{itemize}

\subsection*{Notations}
\addcontentsline{toc}{section}{Notations}
\label{section_notations}
\begin{itemize}
    \item We once and for all fix a prime \(p\) and an algebraic closure \(\FFbar_p\) of \(\FF_p\), and if \(q\) is a power of \(p\), we view \(\FF_q\) as a subfield of \(\FFbar_q\coloneqq \FFbar_p\).
    \item If we have a scheme \(X\) over \(S\) and a morphism \(T\to S\), we write \(X_T\coloneqq X \tensor{S}T\). If \(T=\Spec A\), we also write \(X_A\coloneqq X_T\).
    \item The natural numbers \(\NN\) contain \(0\).
    \item The symbol \(\phi\) will always denote the \(q\)-Frobenius morphism whenever we live over \(\FF_q\). For matrix groups this means raising the matrix entries to the \(q\)-th power and is denoted by \(M^{(q)}\), and we also get an induced Frobenius on Weyl groups, compare \cite[(3.11)]{pink2011algebraiczipdata}, which we will again denote by \(\phi\). 
    \item If we have a group \(G\) and \(x\in G\), we write 
    \(\mathrm{int}(x): G\to G, \quad g\mapsto {}^xg\coloneqq xgx\inv.\) 
    We also write \({}^xA\coloneqq xAx\inv\) for some subset \(A\subseteq G\). 
    \item \(\1_n\) denotes the identity \(n\times n\)-matrix, \(\diag(a_1,\dots,a_n)\) the diagonal \(n\times n\)-matrix with the \(a_i\) on the diagonal and \(\mathrm{antidiag}_n(c)\) the \(n\times n\)-matrix with \(a_{ij}=c\) if \(i+j=n+1\) and \(0\) otherwise.
    \item All reductive groups \(G\) are assumed to be connected. As we basically always consider reductive groups over finite fields, we can choose a maximal torus and a Borel subgroup \(T\subseteq B\subseteq G\), so that we get the Weyl group  with its simple reflections. We write \(\Ru P\) for the unipotent radical of a parabolic subgroup of \(G\), and \(\Par_I\) for the scheme of parabolic subgroups of \(G\) of type \(I\), where \(I\) is a subset of the simple reflections.
    \item Given a group scheme \(G\) over \(L\) and a field extension \(L/K\), we write 
    \[\Res_{L/K}G : \Sch/K\to \Sets,\quad S\mapsto G(S_L).\]
    \item If we have a reductive group \(G\) with its Weyl group \(W\), we fix a set of representatives \(\dot w\in G\) for \(w\in W\) that satisfies \((wv)^\cdot = \dot w \dot v\) if \(\ell(wv)=\ell(w) \ell(v)\).
\end{itemize}

\subsection*{Acknowledgements}
I want to thank my PhD advisor Torsten Wedhorn, who has mentored me throughout my journey of learning algebraic geometry over the last four years. You have been a great supervisor and always gave me a nice intuition.

This project was funded by the Deutsche Forschungsgemeinschaft (DFG, German Research Foundation) TRR 326 Geometry and Arithmetic of Uniformized Structures, project number 444845124.
\newpage
\section{Perverse Sheaves on \texorpdfstring{\(G\)}{G}-Zips}
\label{chapter_perv_sh}

We look at the stack of \(G\)-zips that was introduced in \cite{Pink_2015}. We will mainly be using the description as a quotient stack, so let's recall the necessary vocabulary. In what follows we will only look at connected groups, which simplifies some notations from this paper. So let \(G\) be a connected reductive group over a finite field \(\FF_q\). The stack of \(G\)-zips is a disjoint union of the \(G\)-zips of type \(\mu:\Gm[\FFbar_q]\to G_{\FFbar_q}\), where the disjoint union is taken over conjugacy classes of cocharacters. The stack of \(G\)-zips of type \(\mu\) is called \(G\Zip^\mu\) and can be written as a quotient stack as follows. Let \(\kappa\) be the reflex field of the conjugation class of the cocharacter \(\mu\), and we can assume that \(\mu\) is defined over \(\kappa\), as we work over a finite field. This is a finite field extension of \(\FF_q\). The cocharacter \(\mu\) gives rise to opposite parabolic subgroups \((P_-(\mu),P_+(\mu))\) and a Levi subgroup \(L\coloneqq L(\mu)\coloneqq P_-(\mu)\cap P_+(\mu)\) defined by the condition that \(\Lie(P_-(\mu))\) is the sum of the non-positive weight space of \(\mu\) in \(\Lie(G)\) (resp.\ non-negative for \(\Lie(P_+(\mu)\)). Set
\[P=P_-,\qquad Q=\phi^*(P_+)\qquad\text{and}\qquad M=\phi^*(L)\]
with projections \(P\to L\) and \(Q\to M\), both denoted \(x\mapsto \bar x\). We define 
\[E_\mu \coloneqq \set{(x,y)\in P\times Q}{\phi(\bar x)=\bar y},\]
which acts on \(G\) via \((x,y)\cdot g \coloneqq xgy\inv\). Then we have 
\[G\Zip^\mu = [E_\mu\backslash G],\]
see \cite[Prop 3.11]{Pink_2015}. All of these groups and stacks are defined over \(\kappa\).

We use this description as a quotient stack to describe perverse sheaves on the stack of \(G\)-zips of type \(\mu\). Given a scheme \(X\) over a field, the category of perverse sheaves \(\P(X)\) with \(\QQbar_l\)-coefficients is defined in \cite[III.1.1]{Kiehl_Weissauer_2001}. Then the category of perverse sheaves on a quotient stack \([G\backslash X]\) can be defined as \(G\)-equivariant perverse sheaves on \(X\), as explained in \cite[6.2]{achar2021perverse}: To define this category we need the maps
\[\begin{tikzcd}
    G\times G\times X \rar["\id\times a", shift left=5] \rar["m\times\id"] \rar["\pr_{23}", shift right=5] &[.5em] G\times X \rar["a", shift left] \rar["\pr_2"', shift right] & X,
\end{tikzcd}\]
where \(m\) is the multiplication of \(G\) and \(a\) is the action of \(G\) on \(X\).

\begin{definition}
    The category of equivariant perverse sheaves \(\P_G(X)\) is defined as pairs \((\Fscr,\psi)\), where \(\Fscr\in \P(X)\) and \(\psi: \pr_2^*\Fscr\isom a^*\Fscr\), satisfying the cocycle condition
    \((\id\times a)^*\psi \circ \pr_{23}^*\psi = (m\times \id)^*\psi\).
    The morphisms 
    \((\Fscr,\psi)\to(\Fscr',\psi')\) 
    are those morphisms \(f:\Fscr\to\Fscr'\) in \(\P(X)\) that make the square
    \[\begin{tikzcd}
        \pr_2^*\Fscr \rar["\psi"] \dar["\pr_2^*f"'] & a^*\Fscr \dar["a^*f"]\\
        \pr_2^*\Fscr' \rar["\psi'"] & a^*\Fscr'
    \end{tikzcd}\]
    commute.
\end{definition}

\begin{remark}
    One can also define perverse sheaves on algebraic stacks locally of finite type, which yields for quotient stacks with a smooth affine group an equivalent description to equivariant perverse sheaves, see \cite[5.5]{laszlo2006perverse}.
\end{remark}

\begin{theorem}
\label{thm_simp_perv}
    Given a connected reductive group \(G\) over \(\FF_q\) and a cocharacter \(\mu:\Gm[\FFbar_q]\to G_{\FFbar_q}\), the simple objects of \(\P(G\Zip^\mu_{\FFbar_q})=\P([E_\mu\backslash G]_{\FFbar_q})=\P_{E_{\mu},\FFbar_q}(G_{\FFbar_q})\) are given by pairs \((w,\theta)\), where \(w\) is an orbit of the \(E_\mu\)-action on \(G\) and \(\theta\) is an irreducible \(\QQbar_l\)-representation of the geometric points of the group of connected components of the stabilizer of this group action at \(w\).
\end{theorem}
\begin{proof}
    Note that \(E_\mu\) is connected. Then this is precisely the statement from \cite[B.13]{Clausen2008THESC}. Note that in this paper all the schemes live over \(\CC\), but this hypothesis is not used throughout appendix B, so that the result also holds over any algebraically closed field. For a different reference see \cite[1.1 and 1.4]{JOSHUA2022289}.
\end{proof}

The orbits and the group of connected components of the stabilizers are described in \cite{Pink_2015}, so we introduce the vocabulary we need to compute them. Let \(T\subseteq B\subseteq G_{\FFbar_q}\) be a maximal torus and Borel subgroup. We may and will assume \(T\subseteq M\) and \(B\subseteq Q\).    
Then we obtain the Weyl-Group of \(G\) by
\[W = \Norm_{G(\FFbar_q)}(T(\FFbar_q))/T(\FFbar_q).\]
Let \(S\subseteq W\) be the simple elements with respect to \(T\subseteq B\subseteq G_{\FFbar_q}\), see \cite[21.41]{milne2017algebraic}. We define the length function on \(W\) by
\[\ell: W\to \NN,\quad w\mapsto \min\set{n\in\NN}{\exists s_1,\dots,s_n\in S: w=s_1\cdots s_n}.\]
For subsets \(K,K'\subseteq S\) we write \(W_K\) for the subgroup generated by \(K\) and \({}^K W\) (resp. \(W^{K'}\), resp. \({}^KW^{K'}\)) for the subset of \(w\in W\) that are of minimal length in the coset \(W_K w\) (resp. \(w W_{K'}\), resp. \(W_K w W_{K'}\)). Let \(w_{0,K}\) denote the unique element of maximal length in \(W_K\) and \(w_0\coloneqq w_{0,S}\). The relative Frobenius of \(G\) induces a Frobenius \(\phi:(W,S)\to (W,S)\) of Coxeter pairs, which in particular preserves length. (For \(\GL_n\), the relative Frobenius raises the matrix entries to the \(q\)-th power, so in this case \(\phi=\id_W\).) Let \(I,J\subseteq S\) be the type of \(P,Q\), see \cite[21.91]{milne2017algebraic} (parabolics can be conjugated to contain \(B\)). 
We have \(J={}^{w_0}\phi(I)\). 
Define
\begin{align*}
    y &= w_0 w_{0,I},\\
    % x &= \phi(y) = w_0 \phi(w_{0,I}) = w_0 w_{0,\phi(I)},\\
    \dot w &\in \Norm_{G(\FFbar_q)}(T(\FFbar_q)) \text{ a representative of }w\in W\text{, see \hyperref[section_notations]{Notations}.}
\end{align*}
For \(w\in W\), let \(K_w\) be the largest subset of \(J\cap {}^{w\inv} I\) such that \(\big(\phi\circ\intt(yw)\big)(K_w)=K_w\).

\begin{theorem}[{\cite[3.20]{Pink_2015}}]
    The orbits of \([E_\mu\backslash G]\) are given by \(\Gal(\FFbar_q/\kappa)\backslash {}^IW\) (as topological spaces).
\end{theorem}

\begin{theorem}[{\cite[3.34]{Pink_2015}}]
    The geometric points of the group of connected components of the stabilizer at the orbit \(w\in {}^IW\) is given by
    \[\Pi_w = \set{h\in H_w(\FFbar_q)}{h=\phi\big(\dot y \dot w h (\dot y \dot w)\inv\big)},\]
    where \(H_w\) is the Levi subgroup of \(G_{\FFbar_q}\) of type \(K_w\) containing \(T\).
\end{theorem}

The morphism \(\phi\circ\intt(\dot y\dot w)\) defines a descend datum of \(H_w\) from \(\FFbar_q\) to \(\FF_q\), so \(\Pi_w\) is the set of \(\FF_q\)-valued points of some form of \(H_w\).

\begin{example}[\(P=G\)]
    First we look at the case that \(\mu\) is a central cocharacter, for general \(G\). For example, \(\mu\) can be constant. In this case \(I=S\), hence 
    \[\{\text{orbits}\}={}^IW=\{\id\},\]
    so there is only one orbit. As \(w_0\) always has order 2, we have \(y=w_0^2=\id\), hence \(K_{\id}=S\) and \(H_{\id}=G\), so
    \[\Pi_{\id} = \set{h\in G(\FFbar_q)}{h=\phi(h)} = G(\FF_q),\]
    so in this case the simple perverse sheaves are the irreducible representations of the \(\FF_q\)-valued points of \(G\). These can be studied using Deligne-Lusztig theory.
\end{example}

\begin{example}
    We always have \(w_{0,I}w_0=y\inv\in {}^IW\), which is the longest element of \({}^IW\) and hence corresponds to the open orbit, and at this orbit the group of connected components is 
    \[\Pi_{y\inv} = \set{h\in H_{y\inv}(\FFbar_q)}{h=\phi(h)} = H_{y\inv}(\FF_q).\]
    The type \(K_{y\inv}\) of \(H_{y\inv}\) can be made more explicit. We have 
    \(J\cap {}^{w\inv}I={}^{w_0}\phi(I)\cap {}^{w_0w_{0,I}}I= \phi({}^{w_0}I)\cap {}^{w_0}I\), as \({}^{w_{0,I}}I=I\), and hence
    \[K_{y\inv} = \set{s\in {}^{w_0}I}{\forall i\ge0: \phi^i(s)\in {}^{w_0}I} = \bigcap_{i\ge0}\phi^i({}^{w_0}I).\]
\end{example}

\subsection{Examples for \texorpdfstring{\(\GL_n\)}{GLn}}

We now apply these theorems to the group \(\GL_n\) for several different cocharacters.

\begin{example}[\(\GL_n, Q=B\)]
    Now we take \(G=\GL_n\) and a regular cocharacter, for example \(\mu(t)=\diag(t^n,t^{n-1},\dots,t)\). Then \(Q\) are the upper triangular matrices, \(P\) the lower triangular matrices, hence we must take \(B=Q\) and \(T=L=M\), the diagonal matrices. As the Weyl group we have the permutation group \(\S_n\) with simple reflections given by \((i\;\;i+1)\), and as representatives we always choose the usual permutation matrices, see \cite[21.40]{milne2017algebraic}. Note that multiplication in \(\S_n\) is given by \((\sigma_1\sigma_2)(i)=\sigma_1(\sigma_2(i))\).  The action of the Frobenius on \(W\) is trivial. By the choice of \(\mu\) we have \(I=J=\emptyset\), and hence 
    \[\{\text{orbits}\}={}^IW = W = \S_n.\]
    We have \(K_w=\emptyset\) and hence \(H_w=T\) for all \(w\in W\). We have \(y=w_0=(1\;\;n)(2\;\;n-1)\cdots\), so 
    \begin{align*}
        \Pi_w &= \set{\diag(a_1,\dots,a_n)\in T(\FFbar_q)}{\diag(a_1,\dots,a_n)=\phi\big(\dot w_0\dot w \diag(a_1,\dots,a_n) \dot w\inv\dot w_0\inv\big)}\\
        &= \set{\diag(a_1,\dots,a_n)\in T(\FFbar_q)}{\diag(a_1,\dots,a_n)= \diag(a_{(w_0w)\inv(1)}^q,\dots,a_{(w_0w)\inv(n)}^q)}\\
        &= \bigtimes_{\text{cycles }\sigma\text{ in }w_0w} \set{a\in\FFbar_q^\times}{a=a^{q^{\mathrm{ord}(\sigma)}}}\\
        &= \bigtimes_{\text{cycles }\sigma\text{ in }w_0w} \FF_{q^{\mathrm{ord}(\sigma)}}^{\times}\\
        &= \bigtimes_{\text{cycles }\sigma\text{ in }w_0w} \Res_{\FF_{q^{\mathrm{ord}(\sigma)}}/\FF_q}\Gm(\FF_q),
    \end{align*}
    where the order of a cycle \(\sigma\) is the smallest number \(n>0\) with \(\sigma^n=\id\), and we used that each permutation can be decomposed uniquely (up to reordering) into disjoint cycles with oders adding to \(n\). The third equality comes from the fact that if 1 is part of a cycle of, say, order 3, then \(a_1=a_{(w_0w)\inv(1)}^q=a_{(w_0w)^{-2}(1)}^{q^2}=a_1^{q^3}\).\\
    As these stabilizers are abelian, the number of their irreducible representations equals the order of the group, see \cite[3.2.10, 5.1]{serre1977linear}. For arbitrary finite groups, the number of irreducible representations (over \(\CC\)) equals the number of conjugacy classes, see \cite[2.30]{fultonharris}.
\end{example}

\begin{example}[\(\GL_n\), \((1,n-1)\)]
    Let \(n\ge3\), take \(G=\GL_n\) and consider \(\mu(t)\coloneqq \diag(t,1,\dots,1)\), which corresponds to \(Q=\begin{pNiceMatrix}
* & * & \Cdots & * \\
0 & \Block[borders={top,left, tikz=dashed}]{3-3}{} * & \Cdots & * \\
\Vdots & \Vdots & \Ddots & \Vdots \\
0 & * & \Cdots & *
\end{pNiceMatrix}\). We have \(W=\S_n\) with \(S\) as above, and 
\[I=\{(23),\dots,(n-1\;\;n)\}.\]
We claim 
\[\{\text{orbits}\}={}^IW = \{\id,(12),(123),\dots,(1\cdots n)\}.\]
"\(\supseteq\)": Note that \(W_I\) are precisely those permutations that fix the first element. Given the permutation \((1\cdots k)\) we need to argue that there is no such permutation we can postcompose with to shorten the length. But in our case the length is the number of pairs \(i<j\) with \(w(i)>w(j)\). But after the permutation \((1\cdots k)\), we have \(w\inv(2)<\cdots<w\inv(n)\), so an element in \(W_I\) cannot decrease the length.\\
"\(\subseteq\)": As \({}^IW\) is a system of representatives for the quotient \(W_I\backslash W\), this follows from the inequality
\[n! = \#\mathrm S_n = \#{}^IW \cdot \# W_I \ge \#\{\id,(12),(123),\dots,(1\cdots n)\} \cdot \#\mathrm S_{n-1} = n \cdot (n-1)!=n!,\]
which must therefore be an equality.

Now let \(1\le k\le n\) and set \(w_k\coloneqq (1\cdots k)\). We have 
\begin{align*}
    J &= {}^{w_0}I\\
    &= \{(12),\dots,(n-2\;\;n-1)\}\\[10pt]
    {}^{w_k}I &= \{{}^{w_k}(23),\dots,{}^{w_k}(n-1\;\;n)\}\\
    &= \{(34),\dots,(k-1\;\;k), (k\;\;1), (k+1\;\;k+2),\dots,(n-1\;\;n)\}\\[10pt]
    J\cap {}^{w_k}I &= S\setminus\{(12),(23),(k\;\;k+1),(n-1\;\;n)\}
\end{align*}
We have \(y=w_0 w_{0,I} = (n\;\;n-1\cdots 1)\) and hence \(y w_k = (n\;\;n-1\cdots k)\). Therefore we get 
\[K_{w_k} = \set{(i\;\;i+1)\in S}{3\le i<k-1} = \{(34),\dots,(k-2\;\;k-1)\},\]
and \(H_{w_k}\) looks something like \(\begin{pNiceArray}{ccc@{\;\;\;\;}cc@{\;\;}c}[small]
    * \\
    & * \\
    & & \Block[borders={bottom,top,right,left,tikz=dashed}]{2-2}<\Large>{*} \\
    & \\
    & & & & *\\[5pt]
    & & & & & *
    \CodeAfter \line{5-5}{6-6}
\end{pNiceArray}\). Therefore
\begin{align*}
    \Pi_{w_k} &= \set{h\in H_{w_k}(\FFbar_q)}{h=\phi\big(\dot y \dot w_k h (\dot y \dot w_k)\inv\big)}\\
    &= \set{(a_1,a_2,A,a_k)\in \FFbar_q\times\FFbar_q\times\GL_{k-3}(\FFbar_q)\times \FFbar_q}{a_1^q=a_1, a_2^q=a_2, A^{(q)}=A, a_k^{q^{n-k+1}}=a_k}\\
    &= \FF_q^\times\times\FF_q^\times\times\GL_{k-3}(\FF_q)\times \FF_{q^{n-k+1}}^\times\\
    &= \big(\Gm^2\times \GL_{k-3}\times \Res_{\FF_{q^{n-k+1}}/\FF_q}\Gm\big)(\FF_q).
\end{align*}
\end{example}

\begin{example}[\(\GL_4\), \((2,2)\)]
    Let \(G=\GL_4\) and \(\mu(t)=\diag(t,t,1,1)\), which implies \(I=\{(12),(34)\}\) and \(Q=\begin{psmallmatrix}
        * & * & * & * \\
        * & * & * & * \\
          &   & * & * \\
          &   & * & *
    \end{psmallmatrix}\). One can compute
    \[\{\text{orbits}\}={}^IW = \{\id, (23), (132), (234), (1342), (13)(24)\},\]
    with specialization relations given by 
    \[\begin{tikzcd}
                   &                              & {(132)} \arrow[rd, rightsquigarrow] &                &     \\[-20pt]
    {(13)(24)} \arrow[r, rightsquigarrow] & {(1342)} \arrow[ru, rightsquigarrow] \arrow[rd, rightsquigarrow] &                  & {(23)} \arrow[r, rightsquigarrow] & \id. \\[-20pt]
                   &                              & {(234)} \arrow[ru, rightsquigarrow] &                &    
    \end{tikzcd}\]
    We have \(y=w_0 w_{0,I} = (14)(23)\cdot(12)(34) = (13)(24)\) and \(I=J\). \\
    "\(w=\id\)": \(K_{\id}=I\), \(yw = (13)(24)\) and
    \begin{align*}
        \Pi_{\id} &= \set{\begin{pmatrix}
            A_1 \\
            & A_2
        \end{pmatrix}\in \GL_4(\FFbar_q)}{\begin{pmatrix}
            A_1 \\
            & A_2
        \end{pmatrix} = \begin{pmatrix}
            A_2^{(q)} \\
            & A_1^{(q)}
        \end{pmatrix}}\\
        &= \GL_2(\FF_{q^2}) = \Res_{\FF_{q^2}/\FF_q}\GL_2(\FF_q).
    \intertext{"\(w=(13)(24)\)": \(K_{(13)(24)}=I\), \(yw=\id\) and hence}
    \Pi_{(13)(24)} &= \GL_2(\FF_q)^2
    \intertext{"\(w=(23)\)": \(K_{(23)}=\emptyset\), 
    \(yw=(13)(24)(23)=(1342)\), so} 
    \Pi_{(23)}&=\FF_{q^4}^\times = \Res_{\FF_{q^4}/\FF_q}\Gm(\FF_q).
    \intertext{"\(w=(132)\)": \(K_{(132)}=\emptyset\), 
    \(yw=(13)(24)(132)=(234)\), so}
    \Pi_{(132)}&=\FF_q^\times \times \FF_{q^3}^\times = \big(\Gm \times \Res_{\FF_{q^3}/\FF_q}\Gm\big)(\FF_q).
    \intertext{"\(w=(234)\)": \(K_{(234)}=\emptyset\), 
    \(yw=(13)(24)(234)=(132)\), so} 
    \Pi_{(234)}&=\FF_q^\times \times \FF_{q^3}^\times = \big(\Gm \times \Res_{\FF_{q^3}/\FF_q}\Gm\big)(\FF_q).
    \intertext{"\(w=(1342)\)": \(K_{(1342)}=\emptyset\), 
    \(yw=(13)(24)(1342)=(23)\), so} 
    \Pi_{(1342)}&=(\FF_q^\times)^2 \times \FF_{q^2}^\times = \big(\Gm^2 \times \Res_{\FF_{q^2}/\FF_q}\Gm\big)(\FF_q).
    \end{align*}
\end{example}

\subsection{Symplectic Group}
We now look at the symplectic group \(\Sp_4\), the isometry group of a nondegenerate alternating bilinear form on \(\FF_q^4\), which is unique up to base change. We use results about this group explained in \cite[4.5]{Makisumi2011RedGrpEx}. We take the bilinear form given by 
\[J=\begin{pmatrix}
    & & & 1 \\
    & & 1 \\
    & -1 \\
    -1
\end{pmatrix}\]
and define the group scheme
\[\Sp_4: \Alg{\FF_q}\to \Grp,\quad R\mapsto \set{A\in\GL_4(R)}{A^\top JA=J}.\]
We make this choice of \(M\) because then the standard Borel resp.\ torus of \(\GL_4\), intersected with \(\Sp_4\), will be a Borel \(B\) resp.\ torus \(T\) of \(\Sp_4\). This torus consists of matrices \(t=\diag(t_1,t_2,t_2\inv,t_1\inv)\), and we set \(\alpha_1,\alpha_2:T\to\Gm\) with \(\alpha_i(t)=t_i\) and \(\lambda_1,\lambda_2:\Gm\to T\) with \(\lambda_1(x)=\diag(x,1,1,x\inv)\) and \(\lambda_2(x)=\diag(1,x,x\inv,1)\). They span the \(\ZZ\)-module of (co-)characters. 
    
The Lie algebra consists of matrices \(M\) with \(M^\top J + JM = 0\), i.e.\ 
\(M=\begin{psmallmatrix}
    a & b & c & d\\
    e & f & g & c\\
    i & j & - f & - b\\
    m & i & - e & - a
\end{psmallmatrix}\).
The roots are given by 
\[\Phi = \{\pm2\alpha_1, \pm2\alpha_2, \pm(\alpha_1-\alpha_2), \pm(\alpha_1+\alpha_2)\},\]
due to the computation \(t M t\inv = \begin{psmallmatrix}a & b t_{1} / t_{2} & c t_{1} t_{2} & d t_{1}^{2}\\e t_{2} / t_{1} & f & g t_{2}^{2} & c t_{1} t_{2}\\i / t_{1} t_{2} & j/t_{2}^{2} & - f & - b t_{1} / t_{2}\\ m/t_{1}^{2} & i / t_{1} t_{2} & - e t_{2} / t_{1} & - a\end{psmallmatrix}\), with corresponding unipotent subgroups
\NiceMatrixOptions{cell-space-limits=2pt}
\[\arraycolsep=15pt
\begin{NiceArray}{l|cccc}
\alpha & 2\alpha_1 & 2\alpha_2 & \alpha_1 - \alpha_2 & \alpha_1 + \alpha_2 \\ \hline
U_\alpha &
\begin{psmallmatrix}
1 &   &   & x \\
  & 1 &   &   \\
  &   & 1 &   \\
  &   &   & 1
\end{psmallmatrix} &
\begin{psmallmatrix}
1 &   &   &   \\
  & 1 & x &   \\
  &   & 1 &   \\
  &   &   & 1
\end{psmallmatrix} &
\begin{psmallmatrix}
1 & x &   &   \\
  & 1 &   &   \\
  &   & 1 & x \\
  &   &   & 1
\end{psmallmatrix} &
\begin{psmallmatrix}
1 &   & x &   \\
  & 1 &   & x \\
  &   & 1 &   \\
  &   &   & 1
\end{psmallmatrix} 
\end{NiceArray}\]
and \(U_{-\alpha}=U_\alpha^\top\). 
The cocharacter \(\lambda_\uparrow \coloneqq 2\lambda_1 + \lambda_2\) defines the positive roots
\[\Phi^+ = \set{\alpha\in\Phi}{\langle \alpha, \lambda_\uparrow \rangle >0} = \{2\alpha_1, 2\alpha_2, (\alpha_1-\alpha_2), (\alpha_1+\alpha_2)\}\]
For a cocharacter \(\lambda\) we can define the parabolic subgroup
\[P_\lambda \coloneqq \big\langle T,U_\alpha \,;\, \alpha\in\Phi, \langle \alpha,\lambda\rangle \ge0\big\rangle\]
and get \(P_{\lambda_\uparrow}=B\). The simple roots in \(\Phi^+\) are \(r\coloneqq 2\alpha_2\) and \(s\coloneqq \alpha_1-\alpha_2\). The Weyl group is \(W=\mathrm D_4\), the symmetry group of the square:
\[\begin{array}{l|cccccccc}
    \text{Element} & \id & s & r & rs & sr & rsr & srs & rsrs \\ \hline
    \text{Length} & 0 & 1 & 1 & 2 & 2 & 3 & 3 & 4 \\
    \text{Lift in }\Sp_4 & 
    \scalebox{0.7}{\(\begin{psmallmatrix}
    + &   &   &   \\
      & + &   &   \\
      &   & + &   \\
      &   &   & + \\
    \end{psmallmatrix}\)} &
    \scalebox{0.7}{\(\begin{psmallmatrix}
      & + &   &   \\
    - &   &   &   \\
      &   &   & - \\
      &   & + &   \\
    \end{psmallmatrix}\)} &
    \scalebox{0.7}{\(\begin{psmallmatrix}
    + &   &   &   \\
      &   & + &   \\
      & - &   &   \\
      &   &   & + \\
    \end{psmallmatrix}\)} &
    \scalebox{0.7}{\(\begin{psmallmatrix}
      & + &   &   \\
      &   &   & - \\
    + &   &   &   \\
      &   & + &   \\
    \end{psmallmatrix}\)} &
    \scalebox{0.7}{\(\begin{psmallmatrix}
      &   & + &   \\
    - &   &   &   \\
      &   &   & - \\
      & - &   &   \\
    \end{psmallmatrix}\)} &
    \scalebox{0.7}{\(\begin{psmallmatrix}
      &   & + &   \\
      &   &   & - \\
    + &   &   &   \\
      & - &   &   \\
    \end{psmallmatrix}\)} &
    \scalebox{0.7}{\(\begin{psmallmatrix}
      &   &   & - \\
      & - &   &   \\
      &   & - &   \\
    + &   &   &   \\
    \end{psmallmatrix}\)} &
    \scalebox{0.7}{\(\begin{psmallmatrix}
      &   &   & - \\
      &   & - &   \\
      & + &   &   \\
    + &   &   &   \\
    \end{psmallmatrix}\)} \\[7pt]
    \text{Geometric} & 
    \scalebox{2}{F} &
    \rotatebox[origin=c]{-90}{\reflectbox{\scalebox{2}{F}}} & 
    \rotatebox[origin=c]{180}{\reflectbox{\scalebox{2}{F}}} & 
    \rotatebox[origin=c]{-90}{\scalebox{2}{F}} & 
    \rotatebox[origin=c]{90}{\scalebox{2}{F}} &
    \rotatebox[origin=c]{90}{\reflectbox{\scalebox{2}{F}}} &

    \reflectbox{\scalebox{2}{F}} &  
    \rotatebox[origin=c]{180}{\scalebox{2}{F}}
\end{array}\]
In the matrices \(\pm\) stands for \(\pm1\), and the geometry is in the plane where the \(x\)-axis is \(\alpha_1\) and the \(y\)-axis is \(\alpha_2\).

Now we choose the cocharacter \(\lambda_1+\lambda_2\), which corresponds to the parabolic subgroup
\[P_{\lambda_1+\lambda_2}= \big\langle T,U_\alpha : \alpha\in \{2\alpha_1, 2\alpha_2, \pm(\alpha_1-\alpha_2), (\alpha_1+\alpha_2)\}\big\rangle = \begin{psmallmatrix}
    * & * & * & * \\
    * & * & * & * \\
      &   & * & * \\
      &   & * & *
\end{psmallmatrix}\]
of type \(I=\{s\}= \{\alpha_1-\alpha_2\}\). The Frobenius again acts trivial, so
\[\{\text{orbits}\}={}^IW = \set{w\in W}{\ell(sw)>\ell(w)} = \{\id, r, rs, rsr\}.\]
We have \(y= rsrs\cdot s = rsr\) and \(J={}^{rsrs}\{s\} = \{rsrs\cdot s\cdot srsr\}=\{(rsrs)^2s\}=\{s\}\).

\noindent"\(w=\id\)": \(K_w=\{s\}\) as \(ysy\inv=s\) and 
\begin{align*}
    \Pi_{\id} &= \set{
    \begin{pmatrix} A \\ & B \end{pmatrix} \in \Sp_4(\FFbar_q)}{\begin{pmatrix} A \\ & B \end{pmatrix} = \dot y \begin{pmatrix} A^{(q)} \\ & B^{(q)} \end{pmatrix}\dot y\inv}\\
    &= \set{
    \begin{pmatrix} A \\ & B \end{pmatrix} \in \Sp_4(\FFbar_q)}{\begin{pmatrix} A \\ & B \end{pmatrix} = 
    \begin{psmallmatrix} 
     &   & + &   \\
      &   &   & - \\
    + &   &   &   \\
      & - &   &   \\
    \end{psmallmatrix}
    \begin{pmatrix} A^{(q)} \\ & B^{(q)} \end{pmatrix}
    \begin{psmallmatrix} 
     &   & + &   \\
      &   &   & - \\
    + &   &   &   \\
      & - &   &   \\
    \end{psmallmatrix}}\\
    &= \set{
    \begin{pmatrix} A \\ & B \end{pmatrix} \in \GL_4(\FFbar_q)}{\begin{pmatrix} A \\ & B \end{pmatrix} = 
    \begin{pmatrix} \begin{psmallmatrix}
        1 \\ & -1
    \end{psmallmatrix} B^{(q)} \begin{psmallmatrix}
        1 \\ & -1
    \end{psmallmatrix} \\[5pt] & \hspace{-30pt}\begin{psmallmatrix}
        1 \\ & -1
    \end{psmallmatrix} A^{(q)} \begin{psmallmatrix}
        1 \\ & -1
    \end{psmallmatrix} \end{pmatrix}, 
    B= \begin{psmallmatrix} & 1 \\ 1 \end{psmallmatrix}
    A^{-\top}\begin{psmallmatrix} & 1 \\ 1 \end{psmallmatrix}}\\
    &= \set{A\in \GL_2(\FF_{q^2})}{A = 
    \begin{psmallmatrix}
        & 1 \\ -1
    \end{psmallmatrix} 
    (A^{(q)})^{-\top}
    \begin{psmallmatrix}
        & -1 \\ 1
    \end{psmallmatrix}}\\
    % \\
    % &= \set{
    % A=\begin{pmatrix}
    %     a & b \\ c & d
    % \end{pmatrix}
    % \in\GL_2(\FF_{q^2})}{
    % \begin{pmatrix}
    %     a & b \\ c & d
    % \end{pmatrix}
    % = \frac{1}{\det A^q}
    % \begin{pmatrix}
    %     a^q & b^q \\ c^q & d^q
    % \end{pmatrix}
    % }\qquad (\implies \det A = \det A^{-q})
    &= \mathrm U_2(\FF_q),
\intertext{where the unitary group is defined as in definition \ref{Def_GU}, but with \(c=1\). This last equality can be seen by taking the hermitian form given by \(M=\alpha\begin{psmallmatrix}
    0 & -1 \\ 1 & 0
\end{psmallmatrix}\), where \(\alpha\in \FFbar_q\) with \(\alpha^q=-\alpha\), hence \(\alpha\in\FF_{q^2}\). Then \(A^\dagger=M\inv(A^{(q)})^{\top}M = \begin{psmallmatrix}
        & 1 \\ -1
    \end{psmallmatrix} 
    (A^{(q)})^{\top}
    \begin{psmallmatrix}
        & -1 \\ 1
    \end{psmallmatrix}\) and hence \(AA^\dagger=\1\).
\newline
"\(w=r\)": We have \(yw=rs\) and hence \(yw\cdot s\cdot (yw)\inv=rs\cdot r\cdot sr=srs\not=s\), so \(K_w=\emptyset\) and}
    \Pi_r &= \set{
    \begin{psmallmatrix}
        t_1 \\ & t_2 \\ & & t_2\inv \\ & & & t_1\inv
    \end{psmallmatrix}
    \in \Sp_4(\FFbar_q)}{
    \begin{psmallmatrix}
        t_1 \\ & t_2 \\ & & t_2\inv \\ & & & t_1\inv
    \end{psmallmatrix}
    = \begin{psmallmatrix}
      & + &   &   \\
      &   &   & - \\
    + &   &   &   \\
      &   & + &   \\
    \end{psmallmatrix}
    \begin{psmallmatrix}
        t_1^q \\ & t_2^q \\ & & t_2^{-q} \\ & & & t_1^{-q}
    \end{psmallmatrix}
    \begin{psmallmatrix}
      & + &   &   \\
      &   &   & - \\
    + &   &   &   \\
      &   & + &   \\
    \end{psmallmatrix}\inv
    }\\
    &= \set{
    \begin{psmallmatrix}
        t_1 \\ & t_2 \\ & & t_2\inv \\ & & & t_1\inv
    \end{psmallmatrix}
    \in \Sp_4(\FFbar_q)}{
    \begin{psmallmatrix}
        t_1 \\ 
        & t_2 \\ 
        & & t_2\inv \\ 
        & & & t_1\inv
    \end{psmallmatrix}
    = 
    \begin{psmallmatrix}
        t_2^q \\ & t_1^{-q} \\ & & t_1^{q} \\ & & & t_2^{-q}
    \end{psmallmatrix}\inv
    }\\
    &= \set{t_1\in \FF_{q^4}^\times}{t_1 = t_1^{-q^2}}\\
    &= K(\FF_q),
\intertext{where \(K\) is the kernel of 
\(\Res_{\FF_{q^4}/\FF_q}\Gm\twoheadrightarrow \Res_{\FF_{q^2}/\FF_q}\Gm,\;x\mapsto x\cdot\sigma^2(x)\).\newline
"\(w=rs\)": We have \(yw=r\) and hence \(yw\cdot s\cdot (yw)\inv=r\cdot rs\cdot r=sr\not=s\), so \(K_w=\emptyset\) and}
    \Pi_{rs} &= \set{
    \begin{psmallmatrix}
        t_1 \\ & t_2 \\ & & t_2\inv \\ & & & t_1\inv
    \end{psmallmatrix}
    \in \Sp_4(\FFbar_q)}{
    \begin{psmallmatrix}
        t_1 \\ & t_2 \\ & & t_2\inv \\ & & & t_1\inv 
    \end{psmallmatrix}= 
    \begin{psmallmatrix}
    + &   &   &   \\
      &   & + &   \\
      & - &   &   \\
      &   &   & + \\
    \end{psmallmatrix}
    \begin{psmallmatrix}
        t_1^q \\ & t_2^q \\ & & t_2^{-q} \\ & & & t_1^{-q} 
    \end{psmallmatrix}
    \begin{psmallmatrix}
    + &   &   &   \\
      &   & + &   \\
      & - &   &   \\
      &   &   & + \\
    \end{psmallmatrix}\inv
    }\\
    &= \set{
    \begin{psmallmatrix}
        t_1 \\ & t_2 \\ & & t_2\inv \\ & & & t_1\inv
    \end{psmallmatrix}
    \in \Sp_4(\FFbar_q)}{
    \begin{psmallmatrix}
        t_1 \\ & t_2 \\ & & t_2\inv \\ & & & t_1\inv 
    \end{psmallmatrix}= 
    \begin{psmallmatrix}
        t_1^q \\ & t_2^{-q} \\ & & t_2^q \\ & & & t_1^{-q} 
    \end{psmallmatrix}
    }\\
    &= \FF_q^\times \times \set{t_2\in\FF_{q^2}^\times}{t_2\inv=t_2^q}\\
    &= \big(\Gm\times K\big)(\FF_q),
\intertext{where \(K\) is the kernel of 
\(\Res_{\FF_{q^2}/\FF_q}\Gm\twoheadrightarrow \Gm,\;x\mapsto x\cdot\sigma(x)\).\newline
"\(w=rsr\)": We have \(yw=\id\) and hence \(yw\cdot s\cdot (yw)\inv=s\), so \(K_w=\{s\}\) and}
    \Pi_{rsr} &= \set{
    \begin{pmatrix} A \\ & B \end{pmatrix} \in \Sp_4(\FFbar_q)}{\begin{pmatrix} A \\ & B \end{pmatrix} = \begin{pmatrix} A^{(q)} \\ & B^{(q)} \end{pmatrix}}\\
    &= \GL_2(\FF_q).
\end{align*}

One can also generalize some of these computations to higher dimensional symplectic groups. For \(n\in\NN\) we define 
\(J_n=\begin{pmatrix}
    & \mathrm{antidiag}_n(1)\\
    \mathrm{antidiag}_n(-1)
\end{pmatrix}\)
and define the group scheme
\[\Sp_{2n}: \Alg{\FF_q}\to \Grp,\quad R\mapsto \set{A\in\GL_{2n}(R)}{A^\top J_nA=J_n}.\]
The Weyl group is described in \cite[A.7]{Viehmann_2013}. We have 
\begin{align*}
    W  &= \set{w\in\S_{2n}}{w(i)+w(2n+1-i)=2n+1\text{ for all }i=1,\dots,n} \\
    S&=\{s_1,\dots,s_n\},\quad\text{where}\quad s_i = \begin{cases}
        (i\;\;i+1)(2n-i\;\;2n-i+1), & 1\le i<n\\
        (n\;\;n+1), & i=n
    \end{cases}
\end{align*}
We look at the case \(I=S\setminus\{s_n\}\), so that the corresponding Levi subgroup consists of two blocks of size \(n\). Then we have
\begin{align*}
    w_0 &= (1\;\;2n)(2\;\;2n-1)\cdots(n\;\;n+1)\\
    w_{0,I} &= (1\;\;n)(2\;\;n-1)\cdots(\floor{\frac n2}\;\;\ceil{\frac n2})\cdot (n+1\;\;2n)(n+2\;\;2n-1)\cdots\\
    y &= w_0 w_{0,I} = (1\;\;n+1)(2\;\;n+2)\cdots(n\;\;2n)\\
    \dot y &= \begin{pmatrix}
        & \1_n \\
        -\1_n
    \end{pmatrix}\\
    J &= {}^{w_0}I = I\\
    \{\text{orbits}\} &= {}^IW = \set{w\in W}{w\inv(1)<w\inv(2)<\dots<w\inv(n)} \cong \{0,1\}^n
\end{align*}
For the elements \(\id,y\in{}^IW\) we can again compute the stabilizer:

\noindent"\(w=\id\)": We have \(yw=y\) and hence \(K_{\id}=I\) and
\begin{align*}
    \Pi_{\id} &= \set{
    \begin{pmatrix}
        A \\ & B
    \end{pmatrix}
    \in \Sp_{2n}(\FFbar_q)}{
    \begin{pmatrix}
        A \\ & B
    \end{pmatrix}
    = \dot y 
    \begin{pmatrix}
        A^{(q)} \\ & B^{(q)}
    \end{pmatrix}
    \dot y\inv
    }\\
    &= \set{
    \begin{pmatrix}
        A \\ & B
    \end{pmatrix}
    \in \Sp_{2n}(\FFbar_q)}{
    \begin{pmatrix}
        A \\ & B
    \end{pmatrix}
    = 
    \begin{pmatrix}
        B^{(q)} \\ & A^{(q)}
    \end{pmatrix}
    , B = 
    \begin{pNiceMatrix}
        & 1 \\ 1 
        \CodeAfter\line{1-2}{2-1}
    \end{pNiceMatrix}
    A^{-\top}
    \begin{pNiceMatrix}
        & 1 \\ 1 
        \CodeAfter\line{1-2}{2-1}
    \end{pNiceMatrix}
    }\\
    &= \set{A\in \GL_n(\FF_{q^2})}{A = 
    \begin{pNiceMatrix}
        & 1 \\ 1 
        \CodeAfter\line{1-2}{2-1}
    \end{pNiceMatrix}
    (A^{(q)})^{-\top}
    \begin{pNiceMatrix}
        & 1 \\ 1 
        \CodeAfter\line{1-2}{2-1}
    \end{pNiceMatrix}
    }\\
    &= \mathrm U_n(\FF_q).
\intertext{(The reason why this computation is easier than for \(n=2\) is because we chose a different representative for \(y\). For \(n=2\) we made this choice so the representatives are multiplicative, whereas here we only ever look at a representative of \(y\).)}
\intertext{"\(w=y\)": We have \(yw=\id\) and \(J\cap{}^{w\inv}I=I\), hence \(K_y=I\) and}
\Pi_y &= \set{
    \begin{pmatrix}
        A \\ & B
    \end{pmatrix}
    \in \Sp_{2n}(\FFbar_q)}{
    \begin{pmatrix}
        A \\ & B
    \end{pmatrix}
    =
    \begin{pmatrix}
        A^{(q)} \\ & B^{(q)}
    \end{pmatrix}
    }\\
    &= \GL_n(\FF_q).
\end{align*}

\subsection{Unitary Group with Signature \texorpdfstring{\((1,n-1)\)}{(1,n-1)}}
Let \(q\) be a power of a prime \(p\not=2\). Unitary groups are defined through hermitian vector spaces, so lets first understand these. We consider the field extension \(\FF_{q^2}/\FF_q\) and the non-trivial element of the Galois group \((\phi:x\mapsto \overline x = x^q)\in \Gal(\FF_{q^2}/\FF_q)\). Let \(V\) be a \(\FF_{q^2}\)-vector space and \(\skalp{\cdot}{\cdot}:V\times V\to \FF_{q^2}\) be a non-degenerate hermitian form, i.e.\ bi-additive and
\[\skalp{\alpha v}{w} = \alpha\skalp{v}{w} = \skalp{v}{\overline\alpha w},\quad \skalp{v}{w}=\overline{\skalp{w}{v}} \quad\text{for all }v,w\in V\text{ and }\alpha\in \FF_{q^2},\]
\[\forall v\in V\,\exists w\in V: \skalp{v}{w}\not=0.\]
We call 2 such forms \(\skalp[1]{\cdot}{\cdot}\) and \(\skalp[2]{\cdot}{\cdot}\) isomorphic, if there exists \(A\in \GL(V)\) with \(\skalp[1]{v}{w}=\skalp[2]{Av}{Aw}\) for all \(v,w\in V\). We want to show that all hermitian forms on \(V\) are isomorphic, and for this we need the following

\begin{lemma}
\label{image_q+1_power}
The multiplicative map \(f:\FF_{q^2}\to \FF_{q^2},\, x\mapsto x^{q+1}\) satisfies \(\im f=\FF_q\).
\end{lemma}
\begin{proof}
\(\supseteq\): For \(x\in \FF_{q^2}\) we have \(f(x)^q = x^{q^2+q} = x^{1+q} = f(x)\), hence \(f(x)\in\FF_q\).

\(\subseteq\): 
The cyclic group \(\FF_{q^2}^\times \cong \ZZ/(q^2-1)\) has a unique subgroup of order \(q-1\), namely \(\{q+1,2q+2,\dots\}\), and this then has to be the group \(\FF_q^\times\). The map \(f\) translates on \(\ZZ/(q^2-1)\) to multiplication with \(q+1\), simply because multiplication becomes addition. But from the explicit description it is completely clear that in \(\FF_q^\times=\{q+1,2q+2,\dots\}\subset \FF_{q^2}^\times\) every element is a multiple of \(q+1\).
\end{proof}

\begin{lemma}
All hermitian forms on the \(\FF_{q^2}\)-vector space \(V\) are isomorphic.
\end{lemma}
\begin{proof}
Let \(\skalp{\cdot}{\cdot}\) be a hermitian form on \(V\) and \(e_1,\dots,e_n\) a basis of \(V\). Then the hermitian form is uniquely determined by the matrix \(M=\big(\skalp{e_i}{e_j}\big)_{1\le i,j\le n}\), with \(\skalp{e_i}{e_j} = \overline{\skalp{e_j}{e_i}}\), hence \(\overline M^\top = M\).
We want to transform this hermitian form to the "standard scalar product", i.e.\ find a matrix \(A\in\GL(V)\) such that \(\skalp{Ae_i}{Ae_j}=\delta_{ij}\). 
%We have
%\begin{align*}
%    \skalp{Ae_i}{Ae_j} &= \left( \sum_{k=1}^n a_{ki} e_k, \sum_{l=1}^n a_{lj} e_l \right)\\
%    &= \sum_{k,l=1}^n a_{ki} \skalp{e_k}{e_l} \overline a_{lj}\\
%    &= (A^\top \cdot M \cdot \overline A)_{ij},
%\end{align*}
%with usual matrix multiplication.

We proceed in two steps, namely first diagonalize, and then create 1's everywhere with lemma \ref{image_q+1_power}. First we argue that there exists \(v\in V\) with \(\skalp{v}{v}\not=0\). If we had \(\skalp{v}{v}=0\) for all \(v\in V\), then for all \(v,w\in V\) we have
\[0=\skalp{v+w}{v+w}=\skalp{v}{v}+\skalp{w}{w}+\skalp{v}{w}+\skalp{w}{v}=\skalp{v}{w}+\overline{\skalp{v}{w}}\quad\implies\quad \skalp{v}{w}=-\overline{\skalp{v}{w}}.\]
Now we get for all \(x\in \FF_{q^2}\) 
\[x\skalp{v}{w} = \skalp{xv}{w} = -\overline{\skalp{xv}{w}} = -\overline x \overline{\skalp{v}{w}} = \overline x \skalp{v}{w},\]
but for \(x\not\in\FF_q\) we have \(x\not=\overline x\) and hence \(\skalp{v}{w}=0\), contradiction.

Hence we can choose \(f_1\in V\) with \(\skalp{f_1}{f_1}\not=0\). Now we define \(V_1\coloneqq \FF_{q^2}\cdot f_1\) and look at 
\[V_1^\perp\coloneqq \set{v\in V}{\skalp{v}{f_1}=0}.\]
We observe \(V_1\oplus V_1^\perp\), as for \(v\in V\) and \(x\in \FF_{q^2}\) we have 
\[\skalp{v-xf_1}{f_1}=0 \Leftrightarrow \skalp{v}{f_1} = x \skalp{f_1}{f_1} \Leftrightarrow x = \frac{\skalp{v}{f_1}}{\skalp{f_1}{f_1}},\]
and then \(v = xf_1 + (v-xf_1)\) is the desired unique decomposition. Now we define \(f_2\) in the same way we defined \(f_1\), but with \(V_1^\perp\) replacing \(V\). For this it is crucial that \(\skalp{\cdot}{\cdot}\) remains non-degenerate on \(V_1^\perp\), but that's no problem, since if we take \(v\in V_1^\perp\) and find \(w\in V\) with \(\skalp{v}{w}\not=0\), then we write \(w=w_1+w_1^\perp\) and observe \(\skalp{v}{w_1^\perp} = \skalp{v}{w}\not=0\), as \(\skalp{v}{w_1}=0\). We also observe \(\skalp{f_1}{f_2}=0\), as they come from complementary subspaces. We iterate this process until we have \(f_n\). (Possible confusion: Why do we have \(\skalp{f_1}{f_3}=0\)? That's because \(f_3\) is an element of \(V_1^\perp\), just like \(f_2\).) 

All in all we found \(f_1,\dots,f_n\) with \(\skalp{f_i}{f_i}\not=0\) for all \(i\), and \(\skalp{f_i}{f_j}=0\) for \(i\not=j\). Furthermore we observe \(\skalp{f_i}{f_i}=\overline{\skalp{f_i}{f_i}}\) by definition of the hermitian form, and hence \(\skalp{f_i}{f_i}\in \FF_q\). Now we want to scale the elements \(\skalp{f_i}{f_i}\) to \(1\). For this we take \(x_i\in \FF_{q^2}\) and look at \(f_i'=x_i f_i\). We have 
\[\skalp{f_i'}{f_i'}=\skalp{x_i f_i}{x_i f_i} = x_i\cdot \overline x_i \skalp{f_i}{f_i} = x_i^{q+1} \skalp{f_i}{f_i} \overset{!}{=}1.\]
But lemma \ref{image_q+1_power} precisely allows us to choose such \(x_i\). Then the \(f_i'\) indeed satisfy \(\skalp{f_i'}{f_j'}=\delta_{ij}\), and if we define \(A e_i\coloneqq f_i'\), we see that our hermitian form is isomorphic to the "standard scalar product", as
\[\skalp{Ae_i}{Ae_j} = \delta_{ij} = \skalp[\mathrm{standard}]{e_i}{e_j}.\qedhere\]
\end{proof}
One can also show this lemma using the fact that \(\mathrm H^1(\FF_q,\mathrm U)=1\), where \(\mathrm U\) is some unitary group.

When doing computations with the unitary groups, we want that our tori and Borel subgroups are easy to write down, and hence we fix a basis \(e_1,\dots,e_n\) of \(V\), such that the corresponding matrix \(M\) has 1's on the anti-diagonal, and 0's else. In other words, \(\skalp{e_i}{e_j}=\delta_{(n+1-i),j}\). We extend this form with
\[\skalp{\cdot}{\cdot}: V_R\times V_R\to \FF_{q^2}\tensor{\FF_q}R,\quad \skalp{v\otimes r}{w\otimes s} = \skalp{v}{w}\otimes rs,\]
where \(R\) is an \(\FF_q\)-algebra and \(V_R=V\otimes_{\FF_q}\FF_{q^2}\).

\begin{definition}
\label{Def_GU}
    The group of unitary similitudes \(\GU(V)\) is the algebraic group over \(\FF_q\) with
\begin{align*}
    \GU(V)(R) &= \set{g\in \GL(V_R)}{\exists\, c(g)\in R^\times\, \forall\, v,w\in V_R: \skalp{gv}{gw} = c(g)\skalp{v}{w}}\\
    &= \set{g\in \GL_n(\FF_{q^2}\otimes_{\FF_q} R)}{\exists\, c(g)\in R^\times: g^\dagger g = c(g)\1_n},
\end{align*}
where for a matrix \(g\in \GL_n(\FF_{q^2}\otimes_{\FF_q} R)\) we set \((g^\dagger)_{ij}=\overline g_{n+1-j,n+1-i}\) (with conjugation acting trivial on \(R\)). 
\end{definition}
 
If we set \(\dot w_0\) to be the matrix with 1's on the antidiagonal, we have \(g^\dagger = \dot w_0 \overline g^\top \dot w_0\).

\begin{proposition}
    The group \(\GU(V)\) is a form of \(\GL_n\times \Gm\), more precisely 
    \[\GU(V)_{\FF_{q^2}}=\GL_{n,\FF_{q^2}}\underset{\FF_{q^2}}{\times}\Gm[\FF_{q^2}].\]
\end{proposition}
\begin{proof}
For this we look at the case where \(R\) is even an \(\FF_{q^2}\)-algebra. In this case we have
\begin{comment}
\begin{alignat*}{3}
    V_R = V\otimes_{\FF_q} R &= \FF_{q^2}^n \otimes_{\FF_q} \FF_{q^2} \otimes_{\FF_{q^2}} R &&= \big(\FF_{q^2} \times \FF_{q^2}\big)^n \otimes_{\FF_{q^2}} R &&= R^n \times R^n\\
    v\otimes r &\mapsto v \otimes 1 \otimes r &&\mapsto (v, v) \otimes r &&\mapsto (rv,rv)\\
    & \phantom{\mapsto} (\lambda_i)_{1\le i\le n}\otimes \mu\otimes r &&\mapsto \big( (\lambda_i \mu,\lambda_i \overline \mu)\big)_{1\le i\le n} \otimes r,
\end{alignat*}
\end{comment}
\begin{align*}
    V_R = \FF_{q^2}^n\otimes_{\FF_q} R &= (R\times R)^n\\
    (x_i)\otimes r &\mapsto (x_i r, \overline x_i r),
\end{align*}
induced by the isomorphism
\[\FF_{q^2}\otimes_{\FF_q}\FF_{q^2} \to \FF_{q^2}\times \FF_{q^2},\quad \lambda\otimes \mu \mapsto (\lambda\mu,\overline\lambda \mu).\]
The crucial point now is that the conjugation \(\FF_{q^2}\otimes_{\FF_q}R\ni x\otimes r\mapsto \overline x\otimes r\in\FF_{q^2}\otimes_{\FF_q}R\) translates to a swap of coordinates \(R\times R \ni (r,s)\mapsto (s,r)\in R\times R\).

If we now take a matrix \(g\in \GU(V)(R)\) with entries \(g_{ij}=(a_{ij},b_{ij})\in R\times R\), we have \((g^\dagger)_{ij}=(b_{n+1-j,n+1-i},a_{n+1-j,n+1-i})\), and the condition \(g^\dagger g = c(g)\1_n\) translates to \(B' A=c(g)\) (with \(B' = \dot w_0 B^\top \dot w_0\)) and \(A' B=c(g)\). But the last two equations are equivalent and we observe \(B=c(g)(A')\inv\), so it is enough to remember \(A\) and \(c(g)\). In other words, we have a group isomorphism
\begin{align*}
    \GU(V)(R) &\xrightarrow{\sim} \GL_n(R) \times R^\times\\
    g=(a_{ij},b_{ij}) &\mapsto \big((a_{ij}), c(g)\big).
\end{align*}
This is functorial in \(R\), hence we get an isomorphism of algebraic groups
\[\Psi:\GU(V)_{\FF_{q^2}} \cong \GL_{n,\FF_{q^2}} \underset{\FF_{q^2}}{\times} \Gm[\FF_{q^2}].\qedhere\]
\end{proof}

So let's give a torus and a Borel subgroup of \(\GU(V)\). We define
\begin{align*}
    T(R) &= \set{g\in \GU(V)(R)}{g \text{ is a diagonal matrix}}\\
    B(R) &= \set{g\in \GU(V)(R)}{g \text{ is an upper triangular matrix}}
\end{align*}
One hint that these might define a maximal torus and a Borel, is that the dagger-operation fixes these subgroups, i.e.\ \(T(R)^\dagger\subseteq T(R)\) and \(B(R)^\dagger\subseteq B(R)\): Transposition turns a upper triangular matrix into a lower triangular matrix, inversion fixes lower triangular matrices, and conjugating a lower triangular matrix by \(\dot w_0\) turns it back to a upper triangular matrix. When doing the base change to \(\FF_{q^2}\), then \(T\) and \(B\) just become the usual diagonal torus and upper triangular Borel (times \(\Gm\)) inside \(\GL_n \times \Gm\) via \(\Psi\), so they are indeed a torus and Borel in \(\GU(V)\).

As the Weyl group is defined after base change to the algebraic closure (see \cite{Pink_2015}, \paragraphsymbol 3.5), the Weyl group of \(\GU(V)\) is the same as for \(\GL_n \times \Gm\), namely \(\S_n\). The difference now is that the \(q\)-Frobenius acts non-trivial on the Weyl group. To see this, we first investigate how the Frobenius acts on our group. For an \(\FF_q\)-algebra \(R\) we have the Frobenius \(\phi:R\to R,\, x\mapsto x^q\). For a group \(G\) over \(\FF_q\), the Frobenius is defined on \(R\)-valued points as 
\[\phi: G(R) \to G(R),\quad (\Spec R \to G) \mapsto (\Spec R \xrightarrow{\Spec(\phi)}\Spec R \to G).\]
This definition yields that the Frobenius on \(\GL_n(R)\) just raises all the matrix entries to the \(q\)-th power. So now let's look at \(\GU(V)\). Using the equation
\[\GU(V)(R)=\set{g\in \GL_n(\FF_{q^2}\otimes_{\FF_q} R)}{\exists c(g)\in R^\times: g^\dagger g = c(g)\1_n},\]
we see that the Frobenius is given on the entries of the matrix via the map 
\[\FF_{q^2}\otimes_{\FF_q} R \to \FF_{q^2}\otimes_{\FF_q} R,\qquad x\otimes r \mapsto x\otimes r^q.\]
(Indeed, if \(f:R\to S\) is any morphism of \(\FF_q\)-algebras, then the map \(\GU(V)(R)\to\GU(V)(S)\) is given on entries by \(\id\otimes f\), and \(c(g)\) changes to \(f(c(g)\).) As \(\FFbar_q\) is an \(\FF_{q^2}\)-algebra, we recall the isomorphism \(\FF_{q^2}\otimes_{\FF_q}\FFbar_q \to \FFbar_q\times\FFbar_q,\, x\otimes r\mapsto (xr,\overline x r)\), and see that on \(\GU(V)(\FFbar_q)\) the Frobenius acts entrywise as \(\FFbar_q\times\FFbar_q \ni (a_{ij},b_{ij}) \mapsto (b_{ij}^q,a_{ij}^q)\in\FFbar_q\times\FFbar_q\). Reason:
\[\begin{tikzcd}
    x\otimes r \ar[rrr,mapsto] \ar[ddd,mapsto] &[-20pt] &[20pt] &[-20pt] x\otimes r^q \ar[dd,mapsto] \\[-15pt]
    & \FF_{q^2}\otimes_{\FF_q}\FFbar_q \rar["\phi"] \dar["{\rotatebox[origin=c]{90}{\(\sim\)}}"'] & \FF_{q^2}\otimes_{\FF_q}\FFbar_q \dar["{\rotatebox[origin=c]{90}{\(\sim\)}}"] \\
    & \FFbar_q\times\FFbar_q \rar["{(a,b)\mapsto (b^q,a^q)}"] & \FFbar_q\times\FFbar_q & (x r^q, \overline x r^q)\\[-15pt]
    (xr,\overline x r) \ar[rr,mapsto] & & (\overline x^q r^q, x^q r^q)
\end{tikzcd}\]
Combining this with the formula \(B=c(g)(A')\inv = c(g) \dot w_0 A^{-\top} \dot w_0\) above, we see that the diagram
\[\begin{tikzcd}
    \GU(V)(\FFbar_q) \rar["\phi"] \dar["\Psi"'] & \GU(V)(\FFbar_q) \dar["\Psi"]\\
    \GL_n(\FFbar_q)\times \FFbar_q^\times \rar["\phi"] & \GL_n(\FFbar_q)\times \FFbar_q^\times\\[-20pt]
    \big(A,c(g)\big) \rar[mapsto] & \big( B^{(q)} , c(g)^q \big) = \big( c(g)^q \dot w_0 (A^{(q)})^{-\top} \dot w_0 , c(g)^q \big)
\end{tikzcd}\]
commutes. (Reminder: \((A^{(q)})_{ij} = a_{ij}^q\).) So now it is doable to state the action of \(\phi\) on the Weyl group \(\S_n\) of \(\GU(V)\). For an element \(w\in\S_n\) we denote its usual permutation matrix by \(\dot w\). Note that this explains the mysterious notation \(\dot w_0\). Then an easy set of representatives is given by
\[\S_n \ni w \longmapsto (\dot w,1)\in \GL_n(\FFbar_q)\times\FFbar_q^\times \cong \GU(V)(\FFbar_q).\]
As permutation matrices are orthogonal, we have \(\dot w^{-\top} = \dot w\). As \(\dot w_0\) is the permutation matrix of the longest element in \(\S_n\), namely \(w_0=\begin{psmallmatrix}1&2&\cdots&n \\ n & n-1 & \cdots & 1\end{psmallmatrix}\), we see that \(\phi\) induces the map 
\[\phi: \S_n \to \S_n, \quad w \mapsto w_0 \; w\; w_0.\]
In particular, on simple reflections \(S=\{(12),(23),\dots,(n-1\;\; n)\}\) the Frobenius acts by interchanging \((12)\) with \((n-1\;\; n)\) and so on. This can also be seen more abstractly: The map \(\phi\) has to be nontrivial since otherwise \(\GU(V) \cong \GL_n \times \Gm\), hence it is of this form, since this is the unique nontrivial automorphism of the Coxeter system \((\S_n,S)\).

When looking at \(G\)-zips, then there is the additional data of a cocharacter over a finite extension of \(\FF_q\). As we are interested in signature \((1,n-1)\), we take the cocharacter
\[\mu: \Gm[\FF_{q^2}] \to \GU(V)_{\FF_{q^2}} \cong \GL_{n,\FF_{q^2}} \times \Gm[\FF_{q^2}],\quad t\mapsto 
\left[\begin{pNiceMatrix}[small]
t \\ & 1 \\ & & \Ddots[shorten=6pt] \\ & & & 1
\end{pNiceMatrix},1\right].\]
The corresponding subset of simple reflections is \(I=\{(23),(34),\dots,(n-1\;\;n)\}\). The topological space of \(\GU(V)-\mathrm{Zip}_{\FF_{q^2}}^\mu\) is given by \(\Gamma \backslash \prescript{I}{}{W}\), where \(\Gamma=\Gal(\FFbar_q / \FF_{q^2})\). But \(\Gamma\) is generated by the \(q^2\)-Frobenius, and this operates trivial on \(W\). Hence the topological space is just \(\prescript{I}{}{W}\), and we have 
\[\prescript{I}{}{W} = \{\id,(12),(123),\dots,(1\cdots n)\}.\]
%Begründung: \glqq\(\supseteq\)\grqq{}: Wir beobachten, dass die von \(I\) aufgespannte Gruppe \(W_I\) isomorph zu \(\mathrm S_{n-1}\) ist und als Untergruppe von \(\mathrm S_n\) gerade aus den Permutationen besteht, die die Elemente \(2\) bis \(n\) beliebig permutieren. Unsere Permutationen sind von der Form \((1\cdots i)\) für \(1\ge i\ge n\). Wenden wir diese Permutation auf die Liste \(1,\dots,n\) an, so erhalten wir die Liste \(i,1,2,\dots,i-1,i+1,\dots,n\). Die Länge einer Permutation ist gerade durch die Anzahl der falsch geordneten Paare in einer solchen Liste gegeben. Wir müssen uns also nun die Frage stellen, ob wir die Anzahl der falschen Paare in der Liste \(i,1,\dots,i-1,i+1,\dots,n\) verringern können, wenn wir in dieser Liste alle Elemente bis auf das Erste beliebig permutieren. Dies ist aber offenbar unmöglich, da die Liste \(1,\dots,i-1,i+1,\dots,n\) keine falschen Paare hat. Daher hat das Element \((1\cdots i)\) minimale Länge in \(W_I (1\cdots i)\) und liegt somit in \(\prescript{I}{}{W}\). Die Inklusion \glqq\(\subseteq\)\grqq{} folgt nun aus der Ungleichung 
%\[n! = \#\mathrm S_n = \#\prescript{I}{}{W} \cdot \# W_I \ge \#\{\id,(12),(123),\dots,(1\cdots n)\} \cdot \#\mathrm S_{n-1} = n \cdot (n-1)!\]
The closed subsets of \(\prescript{I}{}{W}\) are precisely those of the form \(\{\id,\dots, (1\cdots k)\}\) for \(0\le k\le n\).

Now we compute the connected components of the stabilizers. 
\begin{align*}
    I &= \{(23),(34),\dots,(n-1\;\;n)\} \\
    J &= {}^{w_0}\phi(I) = I \\
    w_{0,I} &= \max(W_I) = \begin{psmallmatrix}1&2&\cdots& n \\ 1 & n & \cdots & 2\end{psmallmatrix} \\
    y &= w_0 w_{0,I} = (1\;\; n \;\; n-1\; \cdots\; 3\;\; 2)
    % \dot y &= 
    % \begin{pNiceMatrix}[small]
    % 0 & 1 \\
    % \Vdots& \Ddots & \Ddots \\
    %  & & & 1 \\
    % 1 & \Cdots & & 0
    % \end{pNiceMatrix}
\end{align*}
We fix \(1\le k\le n\) and consider \(w=(1\cdots k)\in\prescript{I}{}{W}\). We have
\begin{align*}
    J\cap {}^{w\inv}I 
    &= \set{\sigma\in I}{{}^w\sigma\in I}\\
    &= I \setminus \{(k-1\;\;k),(k\;\;k+1)\}\\
    &= \set{\sigma\in I}{\sigma(k)=k}.
\end{align*}
Now \(K_w\) is the largest subset of \(J\cap {}^{w\inv}I\) stable under \(\phi\circ \intt(yw) = \intt(w_{0,I}w)\). Hence the key is computing \(w_{0,I}w\). We know \(w_{0,I}(1)=1\) and \(w_{0,I}(x)=n+2-x\) for \(x\not=1\). So we have
\[\begin{array}{c|c|c|c|c|c|c|c|c}
    x & 1 & 2 & \cdots & k-1 & k & k+1 & \cdots & n \\
    w(x) & 2 & 3 & \cdots & k & 1 & k+1 & \cdots & n \\
    w_{0,I}w(x) & n & n-1 & \cdots & n+2-k & 1 & n+1-k & \cdots & 2.
\end{array}\]
In particular we witness that \(1\) is mapped to \(n\), which maps to \(2\), back and forth, until \(k\) maps back to \(1\). Hence \(w_{0,I}w\) contains the cycle \((1\;\;n\;\;2\cdots k)\). 
(For \(k=1\) this is the identity, as \(k\) always maps to \(1\), even in the edge cases.) 
So what happens to the rest of the elements? For this we distinguish between \(k\le \frac{n+1}{2}\) and \(k>\frac{n+1}{2}\). In the first case the cycle \((1\;\;n\;\;2\cdots k)\) contains the last \(k-1\) elements, and hence the elements \(M(n,k)\coloneqq\{k+1,\dots,n+1-k\}\) elements are left. (If \(k=\frac{n+1}{2}\), this set is empty.) But \(w\) does not change these elements, and \(w_{0,I}\) exchanges the first and last element of this set, the second with the second to last, and so on. 

Now consider \(k>\frac{n+1}{2}\). We again have the cycle \((1\;\;n\;\;2\cdots k)\), but this now contains the first and last \(n+1-k\) elements. Hence we are left with the elements \(M(n,k)\coloneqq\{n+2-k ,\dots,k-1\}\). (If \(k=\frac{n+2}{2}\), this set is empty.) And again, one observes that \(w_{0,I}w\) flips these elements from front to back. All in all:

For \(k\le\frac{n+1}{2}\): \(w_{0,I}w\) has one cycle containing the first \(k\) and last \(k-1\) elements, and the other \(f(n,k)\coloneqq n+1-2k\) elements are flipped front to back.

For \(k>\frac{n+1}{2}\): \(w_{0,I}w\) has one cycle containing the first and last \(n+1-k\) elements, and the other \(f(n,k)\coloneqq 2k-n-2\) elements are flipped front to back.

Now, for an element \((i\;\;i+1)\in I\) to be in \(K_w\), both \(i\) and \(i+1\) must lie in \(M(n,k)\), because otherwise conjugating with a power of \(w_{0,I}w\) makes this into a permutation of the form \((1\;\;*)\), which is not in \(I\). But if \(i,i+1\in M(n,k)\), then conjugating this with \(w_{0,I}w\) stays inside \(M(n,k)\) and in particular fixes \(k\), so we get
\begin{align*}
    K_w &= \set{(i\;\;i+1)\in I}{i,i+1\in M(n,k)}.
\end{align*}

Hence the Levi subgroup \(H_w\) consists of one block of size \(f(n,k)\), and all the other blocks are of size \(1\) (and of course we still have \(c\)). If we now consider the stabilizer
\[\Pi_w = \set{(A,c)\in H_w(\FFbar_q)\subseteq \GL_n(\FFbar_q)\times \FF_q}{A=c \dot w_{0,I} \dot w (A^{(q)})^{-\top} \dot w\inv \dot w_{0,I}},\]
we see that the block of size \(f(n,k)\) yields the group \(\GU_{f(n,k)}(\FF_q)\) (where \(\GU_0(\FF_q)=\FF_q^\times\)), and the other "blocks" of size 1 get permuted by the cycle \((1\;\;n\;\;2\cdots k)\), yielding that one block determines all the others, together with an equation of the form
\begin{align*}
    a_1 
    &= c a_k^{-q} 
    = \dots 
    = c^* a_2^{\epsilon q^{n-f(n,k)-2}} 
    = c^* a_n^{-\epsilon q^{n-f(n,k)-1}} \\
    &= c^\delta a_1^{\epsilon q^{n-f(n,k)}},
\end{align*}
% \begin{align*}
%     a_1 
%     &= c a_k^{-q} 
%     = \dots 
%     = c^\delta a_2^{\epsilon q^{n-f(n,k)-2}} 
%     = c^{\delta + \epsilon q^{n-f(n,k)-2}} a_n^{-\epsilon q^{n-f(n,k)-1}} \\
%     &= c^{\delta + \epsilon q^{n-f(n,k)-2}- \epsilon q^{n-f(n,k)-1}} a_1^{\epsilon q^{n-f(n,k)}}
% \end{align*}
where \(\epsilon=-1\) if \(k\le\frac{n+1}{2}\) and \(\epsilon=+1\) if \(k>\frac{n+1}{2}\). Furthermore \(\delta=\sum_{i=0}^{n-f(n,k)-1} (-1)^i q^i = \frac{1-(-q)^{n-f(n,k)}}{1+q}\), but this constant does not play a role in the further analysis. Hence we get the following description of \(\Pi_w\):

For \(\epsilon=1\) we get a short exact sequence
\[\begin{tikzcd}[column sep = 45pt]
1 \rar & \Res_{\FF_{q^{n-f(n,k)}}/\FF_q}\Gm(\FF_q) \rar & \Pi_w \rar & \GU_{f(n,k)}(\FF_q) \rar & 1,
\end{tikzcd}\]
and for \(\epsilon=-1\) we get 
\[\begin{tikzcd}[column sep = 65pt]
1 \rar & K(\FF_q) \rar & \Pi_w \rar & \GU_{f(n,k)}(\FF_q) \rar & 1,
\end{tikzcd}\]
where \(K\) is the kernel of \(\Res_{\FF_{q^{2n-2f(n,k)}}/\FF_q}\Gm \twoheadrightarrow \Res_{\FF_{q^{n-f(n,k)}}/\FF_q}\Gm,\;x\mapsto x\cdot \sigma^{n-f(n,k)}(x)\).
% or written differently
% \[\begin{tikzcd}[column sep = 55pt]
% 1 \rar & \mu_{q^{n-f(n,k)}-\epsilon}(\FFbar_q) \times \mathrm U_{f(n,k)}(\FF_q) \rar & \Pi_w \rar & \FF_q^\times \rar & 1,
% \end{tikzcd}\]
% where the right arrow remembers \(c\) and \(\mathrm U_n\) is the same as \(\GU_n\) but with \(c=1\).

% We can also write the (first) sequence explicitly in the two cases:

% For \(k\le\frac{n+1}{2}\):
% \[\begin{tikzcd}[column sep = 62pt]
% 1 \rar & \underbrace{\mu_{q^{2k-1}+1}(\FFbar_q)}_{\cong \ZZ/(q^{2k-1}+1)} \rar & \Pi_w \rar & \GU_{n+1-2k}(\FF_q) \rar & 1.
% \end{tikzcd}\]
% For \(k>\frac{n+1}{2}\):
% \[\begin{tikzcd}[column sep = 58pt]
% 1 \rar & \underbrace{\mu_{q^{2n-2k+2}-1}(\FFbar_q)}_{=\FF_{q^{2n-2k+2}}^\times} \rar & \Pi_w \rar & \GU_{2k-n-2}(\FF_q) \rar & 1.
% \end{tikzcd}\]

\subsection{Pullback to (Truncated) Displays}

We use the definition of (truncated) displays \(G\Disp_\mu^{W_n}\), \(1\le n \le\infty\), from \cite[5.3.7, 5.4]{lau2018higherframesgdisplays}. We have 
\[G\Disp_\mu^{W_1} = G\Zip^{\mu\inv}\]
by \cite[6.2.3]{lau2018higherframesgdisplays}. For \(m\le n\) we have truncation functors \(G\Disp_\mu^{W_n}\to G\Disp_\mu^{W_m}\) (cf. \cite[5.3.8]{lau2018higherframesgdisplays}).
\begin{theorem}
    For \(m\le n<\infty\), the pullback functor 
    \[\P(G\Disp_\mu^{W_m})\to \P(G\Disp_\mu^{W_n})\]
    is fully faithful and preserves simple objects. 
\end{theorem}
\begin{proof}
    By \cite[5.3.2]{lau2018higherframesgdisplays}, one can write displays as a quotient stack
    \[G\Disp_\mu^{W_n} = [G(W_n)/G(\underline W_n)_\mu],\]
    where \(G(W_n)\) and \(G(\underline W_n)_\mu\) are both affine smooth group schemes over \(\FF_q\) by \cite[5.3.7]{lau2018higherframesgdisplays}. Hence our perverse sheaves are again equivariant perverse sheaves. 
    % We split the truncation map in two parts:
    % \[[G(W_{n+1})/G(\underline W_{n+1})_\mu] \to [G(W_n)/G(\underline W_{n+1})_\mu] \to [G(W_n)/G(\underline W_n)_\mu].\]
    As \(G(\underline W_n)_\mu\) is connected (\cite{wedhorntruncatedshtukas}), we have a diagram 
    \[\begin{tikzcd}
        \P\big([G(W_n)/G(\underline W_n)_\mu]\big) \rar \dar[hook] & \P\big([G(W_{n+1})/G(\underline W_{n+1})_\mu]\big) \dar[hook] \\
        \P\big(G(W_n)\big) \rar & \P\big(G(W_{n+1})\big),
    \end{tikzcd}\]
    cf. \cite[6.2.15]{achar2021perverse}. The lower horizontal arrow is fully faithful, see \cite[3.6.6]{achar2021perverse} combined with \cite[5.3.8]{lau2018higherframesgdisplays} and the fact that we have a short exact sequence
    \[\begin{tikzcd}
        0 \rar & \Lie(G) \rar & G(W_{n+1}) \rar & G(W_n) \rar & 1,
    \end{tikzcd}\]
    so that the fibers are connected. By \cite[3.6.9]{achar2021perverse} the lower map preserves simple objects, and an object in \(\P\big([G(W_n)/G(\underline W_n)_\mu]\big)\) is simple if and only if it is simple in \(\P\big(G(W_n)\big)\), as the subcategory is closed under subobjects and quotients by \cite[6.2.15]{achar2021perverse}.
\end{proof}

%=====================================================================

\newpage
\printbibliography
\addcontentsline{toc}{section}{References}
% \newpage
% \bibliographystyle{amsalpha}
% \bibliography{references}
% \nocite{*}

% \newpage
% \input{Erklaerung}

\end{document}